\documentclass[a4paper,12pt]{article}
\usepackage{amsmath,amssymb,amsthm,syntonly,dsfont,color}
\renewcommand{\P}{\mathds P}
\newcommand{\R}{\mathds R}
\newcommand{\E}{\mathds E}
\newcommand{\eps}{\varepsilon}
\newcommand{\ds}{\displaystyle }
\renewcommand{\O}{\mathcal O}
\newcommand{\bO}{\partial\mathcal O}
\newcommand{\bd}{\mbox{\tiny b}}
\newtheorem{theorem}{Theorem}[section]

\newtheorem{Definition}[subsection]{Definition}
\newtheorem{Corollary}[theorem]{Corollary}
\newtheorem{Lemma}[theorem]{Lemma}

\newtheorem{Hypothesis}[theorem]{Hypothesis}

\newtheorem{Remark}{Remark}[section]
\newtheorem{Example}[Remark]{Example}

\numberwithin{equation}{section}

\newcommand{\ter}{(0,T)\times\O\times\Omega}
\def\<#1>{\langle#1\rangle}

\title{
Stochastic differential equations with variable structure
driven by multiplicative Gaussian noise and sliding mode dynamic}
\usepackage{authblk}
\author[1]{Viorel Barbu}
\author[2]{Stefano Bonaccorsi}
\author[2]{Luciano Tubaro} 
\affil[1]{\small University Al. I. Cuza and Institute of Mathematics Octav Mayer,  Iasi, Romania}
\affil[2]{Department of Mathematics, University of Trento, Italy}

\begin{document}
\maketitle
\begin{abstract} This work is concerned with existence of weak solutions to
discontinuous stochastic differential equations driven by multiplicative Gaussian noise and sliding
mode control dynamics generated by stochastic differential equations with variable structure,
that is with jump nonlinearity.
The treatment covers the finite dimensional stochastic systems and the stochastic
diffusion equation with multiplicative noise.
\end{abstract}

\section{Introduction}
We consider here stochastic differential equations of the form
\begin{equation}\label{e1.1}
\begin{array}{l}
dX+A\, X\,dt+f(X)\,dt=B(X)\,dW,\quad t\in(0,T)\\
X(0)=x,
\end{array}
\end{equation}
where $A\colon D(A)\subset H\to H$ is self-adjoint, positive definite such that
$A^{-1+\delta}$ is of trace class for some $\delta\in(0,1)$, $W$ is a cylindrical Wiener process of the form
\begin{equation}\label{e1.2}
W(t)=\sum_{j=1}^\infty \mu_j\, \beta_j(t)\,e_j.
\end{equation}
Here $\{e_j\}$ is an orthonormal basis in $H$, $Ae_j=\lambda_j\,e_j$ and $\{\beta_j\}_{j=1}^\infty$
is a mutually independent system of Brownian motions in a probability space
$\{\Omega,\mathcal F,\P\}$ with filtration $(\mathcal F_t\}_{t\ge0}$.
The operator $f\colon H\to H$ is Borel measurable and locally bounded while
$B\in L(H,\mathcal L^2(H))$ where $\mathcal L^2(H)$ is the space of Hilbert-Schmidt
operators on $H$.

It should be said that under these general conditions equation \eqref{e1.1} is not well posed
except the case of additive noise $(B(X)=I)$ where \eqref{e1.1} has a unique weak (martingale) solution,
see \cite{DFPR}. Equation \eqref{e1.1} has however a unique strong solution if 
$f$ is Lipschitz or accretive and continuous, see \cite{DZ}, or more generally
if $f$ is a maximal monotone graph in $\R\times\R$ with domain $D(f)=\R$ (see \cite{B1}).

Equations of the form \eqref{e1.1} with discontinuous $f$ describe systems with variable structure
and, in particular, closed-loop control systems with ``sliding'' mode behaviour. Here we shall study from this
perspective two special cases.\\The first one is the finite dimensional system
\begin{equation}\label{e1.3}
\begin{array}{l }
dX+f(X)\,dt=\sigma(X)\,dW    \\
X(0)=x  
\end{array}
\end{equation}
where $W$ is a $n$-dimensional Wiener process and $f\in L_{loc}^\infty(\R^n,\R^n)$,
$\sigma\in Lip(\R^n,L(\R^n,\R^n))$.\\The second one is
the stochastic partial differential equation
\begin{equation}\label{e1.4}
\begin{array}{l}
dX-\Delta X\,dt+f(X)\,dt=b(X)\,dW,\quad \text{in }(0,T)\times\O\\
X=0,\hspace{5.6cm}\text{on }(0,T)\times\bO\\
X(0,\xi)=x(\xi),\hspace{4.15cm}\xi\in \O
\end{array}
\end{equation}
in a bounded and open domain $\O\subset \R^d$, $d\ge 1$ with
smooth boundary $\bO$. Here $W$ is a cylindrical Wiener process of the form \eqref{e1.2}
in $H=L^2(\O)$,
where $\{\beta_j\}_{j=1}^\infty$ is a system of independent Brownian motions in a filtered probability space
$(\Omega,\mathcal F,\mathcal F_t,\P)$, $\mu_j\in \R$, $j=1,2,\ldots$ and $\{e_j\}$ is an orthonormal
basis in $L^2(\O)$ to be made precise later on. Here $f\in L_{loc}^\infty(\R)$ and $b\in \text{Lip}_{loc}(\R)$.

Like
in deterministic case, for existence in equation \eqref{e1.4} one must extend it to a multivalued stochastic
equation of the form
\begin{equation}\label{e1.5}
\begin{array}{l}
dX-\Delta X\,dt+F(X)\,dt\ni b(X)\,dW,\quad \text{in }(0,T)\times\O\\
X=0,\hspace{5.6cm}\text{on }(0,T)\times\bO\\
X(0,\xi)=x(\xi),\hspace{4.15cm}\xi\in \O
\end{array}
\end{equation}
where $F\colon\R\to 2^\R$ is the Filippov map associated with $f$,
that is (see \cite{Fil}, \cite{Fil1})
\begin{equation}
\begin{array}{l}
F(r)=[m(f_r),M(f_r)],\qquad \forall r\in\R\\
\raisebox{-4pt}{$\ds m(f_r)=\lim_{\delta\to0} \mathop{\mbox{ess inf}}_{u\in[r-\delta,r+\delta]}f(u)$}\\
\ds M(f_r)=\lim_{\delta\to0} \mathop{\mbox{ess sup}}_{u\in[r-\delta,r+\delta]}f(u).
\end{array}
\end{equation}
Roughly speaking, $F$ is obtained from $f$ by ``filling" the jumps of $f$ in discontinuity points.
If $f\in L^\infty_{\text{\tiny loc}}(\R^n,\R^n)$, where $n\ge 1$, the Filippov map $F\colon \R^n\to2^{\R^n}$
is defined as
\begin{equation}\label{e1.7}
F(r)=\bigcap_{\delta>0}\bigcap_{m(N)=0}\overline{\textrm{conv}f(B_\delta(r)\setminus N)}
\end{equation}
where $m$ is the Lebesgue measure and $B_\delta(r)$ is the ball of centre $r$ and radius $\delta$.
Of course $F(r_0)=f(r_0)$ in all continuity points $r_0$ of $f$.
Then to get existence in \eqref{e1.3} one should replace $f$ by $F$ given by \eqref{e1.7}.
If $f$ is monotone and measurable then $F$ is maximal monotone in $\R^n\times\R^n$
and locally bounded in $\R^n$, 
(see \cite[Proposition 25]{B0}), and so, as
shown in \cite[Theorem 2.2]{B1}, equation \eqref{e1.3} has a unique strong solution (see also \cite{B2}).
In the general case we consider here, the best that we can however expect is only a martingale solution
for \eqref{e1.3} (see Theorem 2.1, in which in general we do not have the uniqueness of the solution).


The main existence result for equation \eqref{e1.3} is established in Section 2, where it's also given a
``sliding mode" type result for this equation.

In Sections 3, 4  and 5 it is studied a similar problem for equation \eqref{e1.5} and also for a
stochastic parabolic system.
\paragraph{Notation}
We use the standard notation for the Sobolev spaces $H^k(\O)$, $k=1,2$, $H_0^1(\O)$ and the Lebesgue
integrable function spaces on $\O\subset\R^n$. The norm of $H_0^1(\O)$ is denoted by $\|\cdot\|_1$ and the norm
of $L^p(\O)$ by $|\cdot|_p$ ($1\le p\le\infty$). The scalar product of $L^2(\O)$ and the duality pairing
between $H_0^1(\O)$ and the dual space $H^{-1}(\O)$ is denoted by the same symbol $\langle\cdot,\cdot\rangle_2$.
We denote by $C([0,T];H)$ the space of all continuous $H$-valued functions on $[0,T]$ and we also refer to
\cite{DZ} for basic results pertaining stochastic processes with values in Hilbert spaces. Finally,
we denote by $C^k_{\bd}(\R)$, $k=0,1$, the space of functions of class $C^k$ on $\R$,
with continuous and bounded derivatives up to order $k$. The norm in $\R$ or $\R^n$ is denoted
by the same symbol $|\cdot|$, the difference being clear from the context.

\section{Weak solution and ``sliding'' mode for the system \eqref{e1.3}}
We shall study here system \eqref{e1.3}
where $W$ is a $n$-dimensional Wiener process, in a probability space $(\Omega,\mathcal F,\P)$
and $f\in L^\infty_{\text{\tiny loc}}(\R^n;\R^n)$, $\sigma\in \mbox{Lip}(\R^n;\mathcal L(\R^n;\R^n))$.\\
We consider the Filippov map $F\colon\R^n\to\R^n$ associated with $f$ which was 
introduced in \eqref{e1.7}.
\begin{Definition}\label{d2.1}
The system $(\Omega,\mathcal F,\P,(\mathcal F_t)_{t\ge0},W,X))$ is said to be a martingale solution to \eqref{e1.3}
if $(\Omega,\mathcal F,(\mathcal F_t)_{t\ge0},\P)$ is a filtered probability space
on which it is defined an $(\mathcal F_t)_{t\ge0}$-Wiener process $W$
and $X$ is an $(\mathcal F_t)_{t\ge0}$-adapted, $\R^n$-valued, continuous process that satisfies $\P$-a.s.
the equation
\begin{equation}\label{eq:2.1}
X(t)+\int_0^t\eta(s)\,ds=x+\int_0^t \sigma(X(s))\,dW(s),\qquad \forall t\ge0.
\end{equation}
Here $\eta\in L^\infty((0,T)\times\Omega)$, for each $T>0$, is an $(\mathcal F_t)_{t\ge0}$-adapted
process such that
\[
\eta\in F(X),\qquad \mbox{a.e. in } (0,\infty)\times\Omega.
\]
\end{Definition}

This definition extends verbatim to infinite dimensional equation \eqref{e1.1}. In literature such a solution
is also called {\bf weak solution}. A martingale solution which is $\bar{\mathcal F}_t ^W$-adapted, where
$\bar{\mathcal F}_t^W$ is the completed natural filtration of $W$ is called strong solution, see \cite{DFPR}.
\\
We have
\begin{theorem}\label{t2.1}
Assume that $f\colon\R^n\to\R^n$ is measurable and that
\begin{equation}\label{eq:2.2}
|f(r)|\le a_1\,|r|+a_2,\qquad \forall r\in\R^n
\end{equation}
where $a_1\ge0$, $a_2\in\R$. Then for each $x\in\R^n$ there is at least one martingale solution
$(\tilde{\Omega},\tilde{{\mathcal F}},\tilde{\P},\tilde{ W},\tilde X)$ to
\eqref{e1.3} which satisfies the estimate
\begin{equation}\label{eq:2.3}
\tilde{\E}|\tilde X(t)|^2 \le C_T\,(|x|^2+1),\quad x\in R^n,\;t\in[0,T].
\end{equation}
\end{theorem}
\begin{proof}
Consider the approximating equation
\begin{equation}\label{e4.6}
\begin{array}{l }
    dX_\eps+f_\eps (X_\eps)\,dt=\sigma(X_\eps)\,dW,\qquad t\in [0,T]   \\
     X(0)=x   
\end{array}
\end{equation}
where $f_\eps$ is 
a smooth approximation of $f$ given by
\begin{equation}\label{e4.7}
f_\eps(r)=\int_{\R^n}f(r-\eps\,\theta)\,\rho(\theta)\,d\theta,\quad \forall\eps>0, r\in \R^n.
\end{equation}
Here $\rho\in C^\infty_{0}(\R^n)$ is any mollifier such that
\begin{equation}\label{e2.6n}
\rho(r)\ge0,\quad\rho(r)=\rho(-r),\quad \rho(r)=0\mbox{ for }|r|\ge 1,
\quad \int_{-\infty}^\infty\rho(r)\,dr=1.
\end{equation}
Let $X_\eps\in L^2(\Omega;C([0,T];\R^n))$ be the strong solution to \eqref{e4.6}.
By \eqref{eq:2.2} and It\^o's formula it follows that
\[
d\,|X_\eps(t)|_2^2\le C\,|X_\eps(t)|_2^2\,dt+X_\eps(t)\cdot \sigma(X_\eps(t))\,dW_t
\]
and so by the Burkholder-Davis-Gundy theorem (see e.g., \cite{DZ}) we have
\begin{equation}\label{e2.7n}
\E\Big[\sup_{t\in [0,T]} |X_\eps(t)|^2\Big]\le C\,(1+|x|^2),\qquad \forall \eps>0
\end{equation}
(Here and everywhere in the following we shall denote by $C$ several positive constants independent of $\eps$.)

We set $Y_\eps=(X_\eps,W)$ and
we consider $\nu_\eps=\mathcal L(Y_\eps)$ (the law of $Y_\eps$) that is
$\nu_\eps(\Gamma)=\P[Y_\eps\in \Gamma]$ for each Borelian set $\Gamma\subset C([0,T];\R^n)\times C([0,T];\R^n)$. Let us show that
$\{\nu_\eps\}$ is tight in $(C([0,T];\R^n))^2 = C([0,T];\R^n)\times C([0,T];\R^n)$.
This means that for each $\delta>0$ there is a compact subset $\Gamma$ of
$(C([0,T];\R^n))^2$
such that $\nu_\eps(\Gamma^c)\le \delta$ for all $\eps>0$. We take for $r>0$, $\gamma>0$,
\[
\begin{split}
\Gamma=B_{r,\gamma}=\{y\in {(C([0,T];\R^n))^2}: |y(t)|\le r, \forall t\in[0,T],\hspace{2cm}\\
|y(t)-y(s)|\le \gamma \,|t-s|^{\tfrac12},\quad \forall t,s\in[0,T]\}
\end{split}
\]
Clearly, by the Ascoli-Arzel\`a theorem, $B_{r,\gamma}$ is compact in {$(C([0,T];\R^n))^2$}.
On the other hand, by \eqref{e2.4} we have via It\^o's formula applied to the process
$t\to|X_\eps(t)-X_\eps(s)|_2^2$
\[
\begin{split}
\tfrac12\,\E |X_\eps(t)-X_\eps(s)|^2
+
\E\int_s^t\langle f_\eps(X_\eps(\theta)),X_\eps(\theta)-X_\eps(s)\rangle \,d\theta\\
\le C\,\E\int_s^t|X_\eps(\theta)|^2\,d\theta
\quad 0\le s\le t\le T.
\end{split}
\]
Taking into account estimate \eqref{e2.7n}, we obtain via Gronwall's lemma that
\begin{equation}\label{e2.16n}
\E|X_\eps(t)-X_\eps(s)|^2\le C\int_s^t |X_\eps(\theta)|^2\,d\theta\le C\;|t-s|.
\end{equation}
By estimates \eqref{e2.7n}, \eqref{e2.16n}  and by
\[
\rho\;\P[|Y|\ge\rho] \le \E|Y|,\quad \forall \rho>0,
\]
we see that there are $\gamma$, $r$ independent of $\eps$ such that
$\nu_\eps(B_{r,\gamma}^c)\le\delta$, as desired.
Then by the Skorohod's representation theorem there exist a probability space $(\tilde\Omega,\tilde{\mathcal F},\tilde\P)$
and random variables $\tilde X$, $\tilde X_\eps$, $\tilde W_\eps$, $\tilde W$ such that
${\cal L}(\tilde X_\eps,\tilde W_\eps)={\cal L}(X_\eps, W_\eps)$ and for $\tilde\P$-almost every $\omega\in \tilde \Omega$
\begin{equation}\label{e4.8}
\begin{array}{l}
\tilde W_\eps\to \tilde W,\;\tilde X_\eps\to \tilde X\quad\quad \P\mbox{-a.s. in } C([0,T];\R^n)\\
\sigma(\tilde{X_\eps})\to \sigma(\tilde X)\qquad\qquad \P\mbox{-a.s. in } C([0,T];\R^n)
\end{array}
\end{equation}
as $\eps\to 0$. We have also $\mathcal L(f_\eps(\tilde X_\eps))=\mathcal L(f_\eps(X_\eps))$
and so by \eqref{eq:2.2}, \eqref{e4.7} it follows that on a subsequence,  denoted $\{\eps_n\}$,
\begin{equation}\label{e4.9}
f_{\eps_n}(\tilde X_{\eps_n})\to \tilde \eta\qquad \mbox{weak-star in } L^\infty((0,T)\times\tilde \Omega).
\end{equation}
Let us show that
\begin{equation}\label{e4.10}
\tilde \eta\in F(\tilde X),\qquad \mbox{a.e. in } (0,T)\times\tilde\Omega.
\end{equation}
We have by \eqref{e4.7},
\[
\begin{split}
f_\eps(\tilde X_\eps(t,\omega))=
\int_{\R^n}f(\tilde X_\eps(t,\omega)-\eps\,\theta)\,\rho(\theta)\,d\theta
\in \overline{{\textrm{conv}f(B_\eps(\tilde X_\eps(t,\omega)))}},\qquad\\
\forall (t,\omega)\in [0,T]\times\tilde\Omega,
\end{split}
\]
and this implies that
\[
\Sigma(t,\omega)=\big\{\lim_{\eps_n\to 0}f_{\eps_n}(\tilde X_\eps(t,\omega))\big\}
\subset F(\tilde X(t,\omega)),\qquad\forall (t,\omega)\in [0,T]\times\tilde\Omega.
\]
By \eqref{e4.9}, which implies of course also the weak convergence in $L^2((0,T)\times\tilde\Omega)$,
 it follows by Mazur's theorem (see e.g., \cite[pag. 120]{Y}) that there is a convex combination
of $f_{\eps_n}$, that is
$$\varphi_n(t,\omega)=\sum_{i=1}^{k_n} \alpha_i^{(n)}f_{\eps_i}(\tilde X_{\eps_i}(t,\omega)),$$
$\sum_{i=1}^{k_n} \alpha_i^{(n)}=1$, $0\le\alpha_i^{(n)}\le1$, which is strongly convergent in
$L^2((0,T)\times\tilde\Omega)$ to $\tilde\eta$ and so on a subsequence again denoted $\{n\}$,
\[
\lim_{n\to\infty} \varphi_n(t,\omega)=\tilde\eta(t,\omega),\qquad \mbox{a.e. }(t,\omega)\in(0,T)\times\tilde\Omega.
\]
Since $\lim_{n\to\infty} \varphi_n(t,\omega)\in F(\tilde X_\eps(t,\omega))$ we obtain \eqref{e4.10}  as claimed.

If we define
\[
\tilde{\mathcal F}_t^\eps=\sigma\big(\tilde X_\eps(s),\tilde W_\eps(s);0\le s\le t\big),\qquad t\ge 0,
\]
\[
\tilde{\mathcal F}_t=\sigma\big(\tilde X(s),\tilde W(s);0\le s\le t\big),\qquad t\ge 0,
\]
then it follows that $(\tilde W_\eps,\tilde{\mathcal F}^\eps_t)$ and
$(\tilde W,\tilde{\mathcal F}_t)$ are Wiener processes and that $\P$-a.s.,
\[
\tilde X_\eps(t)+\int_0^t f_\eps(\tilde X_\eps(s))\,ds=
x+\int_0^t \sigma(\tilde X_\eps(s)\,d\tilde W_\eps(s),\qquad \forall t\in[0,T].
\]
Taking into account \eqref{e4.9} and that $\P$-a.s (see Lemma 3.1 in \cite{GK})
\[
\lim_{\eps\to0}\int_0^t\sigma(\tilde X_\eps(s))\,d\tilde W_\eps(s)=\int_0^t\sigma(\tilde X(s))\,d\tilde W(s),\qquad \forall t\in[0,T],
\]
we obtain that $\tilde\P$-a.s.
\[
\tilde X(t)+\int_0^t \tilde\eta(s)\,ds=x+\int_0^t \sigma(\tilde X(s))\,d\tilde W(s),\qquad \forall t\in[0,T].
\]
This means that the system
$({\tilde \Omega},{\tilde{\mathcal F}},\{{\tilde{\mathcal F}}_t\}_{t\ge 0},{\tilde \P},{\tilde W}(t),\tilde X(t))$
is a martingale solution to \eqref{e1.3}.
The estimate \eqref{eq:2.3} follows by
 \eqref{e2.7n} which in turn implies that
\[
\E|\tilde X_\eps(t)|^2 \le C\,(1+|x|^2),\quad \forall \eps>0,\;t\in[0,T].
\]
Such a process $\tilde X$ can be extended to all of $(0,\infty)$.
\end{proof}
\begin{Remark}\em
If $f$ and $\sigma$ are in $L^\infty(\R^n)$ and
\begin{equation}
\sum_{i,j=1}^n(\sigma^\ast\sigma)_{ij}(x)\,\xi_i\xi_j\ge \alpha\sum_{i=1}^n \xi_i^2,\qquad
\forall \xi=(\xi_i)\in\R^n,\label{e2.11m}
\end{equation}
then, as shown by A. Yu. Veretennikov \cite{V}, equation \eqref{e1.3} has a unique strong solution $X$.
(On these lines see also \cite{GK}.) It should be said however that for the applications
we have in mind, the nondegeneracy condition \eqref{e2.11m} is too restrictive.
\end{Remark}

The sliding mode dynamics arises in differential systems with variable structure of the form \eqref{e1.3} and
a typical case is that when $f$ has the form
\begin{equation}\label{e2.11n}
f(r)=
\begin{cases}
f_1(r)&\mbox{if } g(r)>0\\
f_2(r)&\mbox{if } g(r)<0
\end{cases}\qquad \forall r\in\R
\end{equation}
where $g\in C^2(\R^n)$, $f_1$,$f_2\in C^1(\R^n, \R^n)$ and $\sigma\in\mbox{Lip}(\R^n,\mathcal L(\R^n,\R^n))$
satisfy the following conditions
\begin{align}
\label{e4.12}&|f_i(r)|\le a_{1i}\,|r|+a_{2i},\quad\forall r\in\R^n,\quad i=1,2\\
\label{e4.13}&\sup_{r\in\R^n}\{|\nabla g(r)|+|D^2g(r)|\}<\infty\\
\label{e4.14}&\nabla g(r)\cdot f_1(r)\ge \alpha\quad\mbox{in }\{r\in\R^n :g(r)>0\}\\
\label{e4.15}&\nabla g(r)\cdot f_2(r)\le -\alpha\quad\mbox{in }\{r\in\R^n :g(r)<0\}\\
\label{e4.16}&|D^2g(r)||\sigma(r)|^2\le C^\ast\,|g(r)|,\qquad \forall r\in\R^n
\end{align}
where $\alpha>0$.  We have
\begin{theorem}\label{t4.2}
Under assumptions \eqref{e4.12}--\eqref{e4.16} for each $x\in \R^n$ there is
a martingale solution
$({\tilde \Omega},{\tilde {\mathcal F}},{\tilde \P},{\tilde W},\tilde X)$
to \eqref{e1.3} with the following properties:\\
(i) if $g(x)=0$ then ${\tilde\P}$-a.s. $g(\tilde X(t))=0,\quad \forall t\ge0$;\\
(ii) if $g(x)\not=0$ and $\tau=\inf\{t>0: g(\tilde X(t))=0\}$ then
\begin{equation}\label{e4.16a}
{\tilde\P}(\tau>t)\le \frac{\tilde C}{\alpha}\,(1-e^{-\tilde C\,t})^{-1}\,|g(x)|,\qquad \forall t>0,
\end{equation}
where $\tilde C=C_1\,C^\ast$, $C_1$ a positive constant independent
of $g$ and $\sigma$. If $C^\ast=0$ then
\[
{\tilde\P}(\tau>t)\le (\alpha\,t)^{-1}|g(x)|,\qquad \forall t>0.
\]
\end{theorem}
Theorem \ref{t4.2} amounts to say that the manifold $\Sigma=\{x: g(x)=0\}$
is invariant for stochastic system \eqref{e1.3} with $f$ given by \eqref{e2.11n} and
that for $x\notin \Sigma$ the solution $\tilde X$ have reached the manifold $\Sigma$ by
time $t$ with a probability greater or equal to $1-(\alpha\,t)^{-1}\,|g(x)|$. In the classical
automatic control terminology (see, e.g., \cite{U}) this means that $g(x)=0$ is a
``sliding mode" equation for system \eqref{e1.3} and $\Sigma$ is a switching surface
for this system. As a matter of fact this is typical ``sliding" mode behaviour for
the solution $X=\tilde X(t)$ and its dynamics has two phases: the first phase
is on time interval $(0,\tau)$ until $X$ reaches surface $\Sigma$ and the second one for
$t\ge\tau$ in which $X(t)$ evolves on sliding surface $\Sigma$. The reaching time
$\tau=\tau(\omega)$ is a stopping time determined by \eqref{e4.16a}.
\renewcommand{\proofname}{Proof of Theorem \ref{t4.2}}
\begin{proof}
We note first that the function $f$ can be written as
\[
f(r)=f_1(r)H(g(r))+f_2(r)H(-g(r)),\qquad \forall r\in\R^n
\]
where $H$ is the Heaviside function while the corresponding Filippov multivalued
function $F$ (see \eqref{e1.7}) is just
\[
F(y)=f_1(y)\tilde H(g(y))+f_2(y)\tilde H(-g(y)),\qquad \forall y\in\R^n
\]
and $\tilde H$ is the multivalued Heaviside function
\begin{equation}\label{eq:2.22}
\tilde H(r)=\begin{cases}
1&\mbox{for } r>0\\
[0,1]&\mbox{for } r=0\\
0&\mbox{for } r<0.
\end{cases}
\end{equation}
In the following we shall use the notations
\[
\text{sgn}(r)=H(r)-H(-r)=\begin{cases}1&\text{for }r>0\\-1&\text{for }r<0\end{cases}
\]
\[
\widetilde{\text{sgn}}(r)=\tilde H(r), \quad \forall r\in\R.\]
Let $\tilde X$ be the
martingale solution to \eqref{e1.3} given by \eqref{e4.8}
where $f$ is as in \eqref{e2.11n}. In order to prove the theorem we need
a few apriori estimates on the solution $X_\eps$ to
\eqref{e4.6} which will be obtained by applying It\^o's formula to the
function $\phi_\lambda(u)=\varphi_\lambda(g(u))$, where $\varphi_\lambda\in C^2(\R)$
is
\begin{equation}\label{e4.17}
\begin{cases}
\varphi_\lambda(0)=0\\
\varphi_\lambda'(y)=\frac{1}{\lambda}y&\mbox{for }|y|\le \lambda,\;\lambda>0\\
\varphi_\lambda'(y)=1+\lambda&\mbox{for }y\ge 2\lambda\\
\varphi_\lambda'(y)=-1-\lambda&\mbox{for }y\le -2\lambda\\
|\varphi_\lambda''(y)|\le \frac{C}{\lambda},&\mbox{for }|y|\le 2\lambda.
\end{cases}
\end{equation}
On $(0,\infty)$ such a function can be taken 
as 
\[
\varphi_\lambda(y)=
\begin{cases}
\frac1{2\lambda}\,y^2&\mbox{for }0\le y\le \lambda\\
(1+\lambda)\,y&\mbox{for }2\lambda\le y<\infty\\
p_\lambda(y)&\mbox{for }\lambda\le y<2\lambda
\end{cases}
\]
where $p_\lambda$ is a fourth order polynomial conveniently chosen and extend it by simmetry on $(-\infty,0)$. 
As a matter of fact, $\varphi_\lambda$ is a smooth approximation of function
$y\to|y|$ and as easily seen
\[
|\varphi_\lambda'(y)-(\widetilde{\mbox{sgn}})_\lambda(y)|\le C\,\lambda,\qquad \forall y\in\R,y\not=0,
\]
where $(\widetilde{\mbox{sgn}})_\lambda$ is the Yosida approximation of $\widetilde{\text{sgn}}$, that is
\[
(\widetilde{\mbox{sgn}})_\lambda(r)=\begin{cases}
\frac1{\lambda}\,|r|&\mbox{for } |r|<\lambda\\
1&\mbox{for } r>\lambda\\
-1&\mbox{for } r<-\lambda.
\end{cases}
\]
We have therefore for all $r\in\R\setminus \{0\}$
\[
\lim_{\lambda\to0} \varphi_\lambda'(r)=\text{sgn}(r)
\]
Taking into account that, $\forall u,v\in\R^n$, one has
\begin{align*}
\label{e4.18}&\nabla\varphi_\lambda(u)=\varphi_\lambda'(g(u))\nabla g(u),\\
&D^2\varphi_\lambda(u)(v)=\varphi_\lambda''(g(u))(\nabla g(u)\cdot v)\nabla g(u)+
\varphi_\lambda'(g(u))D^2g(u)(v)
\end{align*}
 we obtain that
\begin{equation}
\begin{aligned}
\label{e4.20}
d\varphi_\lambda(g(X_\eps(t))) &+\varphi_\lambda'(g(X_\eps(t)))f_\eps(X_\eps(t))\cdot
\nabla g(X_\eps(t))\,dt \\ 
=& \tfrac12\,\mbox{Tr}[\sigma^\ast(X_\eps(t))\sigma(X_\eps(t))D^2\phi_\lambda(X_\eps(t))]\,dt \\
&+\sigma(X_\eps(t))\,dW(t)\cdot\nabla\phi_\lambda(X_\eps(t)),
\qquad t\in [0,T].
\end{aligned}
\end{equation}
Now, taking into account that on $\{g(X_\eps(t))\ne 0\}$
\[
\lim_{\lambda\to0} \varphi_\lambda'(g(X_\eps(t)))=\text{sgn}(g(X_\eps(t)))
\]
and that
in virtue of \eqref{e4.16} and \eqref{e4.17},
\[
\begin{array}{ l}
\varphi_\lambda''(g(X_\eps))=0 \quad  \mbox{on } \{|g(X_\eps)|>2\lambda\} \\
|\varphi_\lambda'(g(X_\eps))|\le 2  \quad  \mbox{on } \{|g(X_\eps)|>2\lambda\}  \\
|D^2g(X_\eps)| |\sigma(X_\eps)|^2\le C^\ast\,|g(X_\eps)|   \quad \mbox{in }(0,T)\times\O\\
|\varphi_\lambda''(g(X_\eps))|\le C\,\lambda^{-1},
\qquad \forall\lambda>0,   
\end{array}
\]
letting $\lambda\to0$ in \eqref{e4.20}, we obtain that
\begin{equation}\label{e4.21}
\begin{aligned}
d|g(X_\eps(t))| &+ f_\eps(X_\eps(t))\cdot\nabla g(X_\eps(t))\,\mbox{sgn}(g(X_\eps(t))
\mathds1_{[|g(X_\eps(t))|>0]}\,dt
\\
=& \tfrac12\,\mbox{Tr}[\sigma^\ast(X_\eps(t))\sigma(X_\eps(t))D^2g(X_\eps(t))]\mbox{sgn}(g(X_\eps(t))\,dt
\\
&+ \sigma(X_\eps(t))dW(t)\cdot\nabla g(X_\eps(t))
\end{aligned}
\end{equation}
By \eqref{e4.7} and \eqref{e2.11n} we have
\[
\begin{aligned}
f_\eps(X_\eps(t))&\cdot\nabla g(X_\eps(t))\,\mbox{sgn}(g(X_\eps(t)))
\\
=& \int\limits_{[g(X_\eps(t)-\eps \theta)>0]}f_1(X_\eps(t)-\eps \theta)\cdot\nabla g(X_\eps(t))\rho(\theta)\,d\theta
\\
&+
\int\limits_{[g(X_\eps(t)-\eps \theta)<0]}f_2(X_\eps(t)-\eps \theta)\cdot\nabla g(X_\eps(t))\rho(\theta)\,d\theta
\end{aligned}
\]
and so taking into account \eqref{e4.12}--\eqref{e4.14} we get
\[
f_\eps(X_\eps(t))\cdot\nabla g(X_\eps(t))\,\mbox{sgn}(g(X_\eps(t)))\ge
\alpha-\delta(\eps)(1+|X_\eps(t)|)\quad \forall t\ge0
\]
where $\delta(\eps)\to 0$ as $\eps\to0$. Taking into account \eqref{e4.21} this yields
\[
\begin{split}
d|g(X_\eps(t))| &+ \alpha\,\mathds1_{[|g(X_\eps(t))|>0]}(1-\delta(\eps)|X_\eps(t)|)\,dt
\\
\le& \tfrac12\,\mbox{Tr}[\sigma^\ast(X_\eps(t))\sigma(X_\eps(t))D^2g(X_\eps(t))]\mbox{sgn}(g(X_\eps(t)))\,dt
\\
&+ 
\mathds1_{[|g(X_\eps(t))|>0]}|g(X_\eps(t))|^{-1}\nabla g(X_\eps(t))\cdot\sigma(X_\eps(t))\,dW(t)
\end{split}
\]
and therefore, for $0\le s\le t\le T$, we have $\P$-a.s.
\[
\begin{split}
|g(X_\eps(t))| &+ \alpha\,\int_s^t\mathds1_{[|g(X_\eps(\theta))|>0]}(1-\delta(\eps)|X_\eps(\theta)|)\,d\theta
\\
\le& |g(X_\eps(s))|+\int_s^t\mathds1_{[|g(X_\eps(\theta))|>0]}
|g(X_\eps(\theta))|^{-1}\nabla g(X_\eps(\theta))\cdot\sigma(X_\eps(\theta))\,dW(\theta)
\\
&+ \tfrac12\int_s^t\mbox{Tr}[\sigma^\ast(X_\eps(\theta))\sigma(X_\eps(\theta))D^2g(X_\eps(\theta))]\mbox{sgn}((g(X_\eps(\theta)))\,d\theta.
\end{split}
\]
The same inequality remains of course true for $(\tilde X_\eps,\tilde W)$ and so letting
$\eps\to 0$ we get that for $\tilde X$ given by \eqref{e4.8}, we have
\[
\begin{split}
|g(\tilde X(t))| &+ \alpha\int_s^t \mathds1_{[|g(\tilde X(\theta))|>0]}\,d\theta
\\
\le& |g(\tilde X(s))|+\tfrac12\int_s^t\mbox{Tr}[\sigma^\ast(\tilde X(\theta))\sigma(\tilde X(\theta))
D^2g(\tilde X(\theta))]\mbox{sgn}((g(\tilde X(\theta)))\,d\theta
\\
&+ \int_s^t \mathds1_{[|g(\tilde X(\theta))|>0]}\nabla g(\tilde X(\theta))\cdot\sigma(\tilde X(\theta))\,dW(\theta),
\quad 0\le s\le t<\infty,\;\tilde{\tilde\P}\mbox{-a.s.}
\end{split}
\]
Taking into account \eqref{e4.16} we get
\begin{multline*}
|g(\tilde X(t))|+\alpha\int_s^t \mathds1_{[|g(\tilde X(\theta))|>0]}\,d\theta 
\\
\le \tilde C\int_s^t |g(\tilde X(\theta))|\,d\theta+
\int_s^t \mathds1_{[|g(\tilde X(\theta))|>0]}\nabla g(\tilde X(\theta))\cdot\sigma(\tilde X(\theta))\,dW(\theta),
\end{multline*}
and so by the Gronwall lemma
\begin{equation}\label{e4.22}
\begin{split}
e^{-\tilde C\,t}|g(\tilde X(t))|+\alpha\int_s^t e^{-\tilde C\,\theta}\mathds1_{[|g(\tilde X(\theta))|>0]}\,d\theta\le
e^{-\tilde C\,s}|g(\tilde X(s))|+\hspace{1cm}\\
\int_s^t e^{-\tilde C\,\theta}\mathds1_{[|g(\tilde X(\theta))|>0]}\nabla g(\tilde X(\theta))\cdot\sigma(\tilde X(\theta))\,dW(\theta),
\quad 0\le s\le t<\infty.
\end{split}
\end{equation}
In particular, it follows by \eqref{e4.22} that if $g(x)=0$ then $g(\tilde X(t))=0$ $\tilde{\P}$-a.s. for all $t\ge 0$.
Moreover, by \eqref{e4.22} it follows that $Z(t)=|g(\tilde X(t))|e^{-\tilde C \,t}$ is a nonnegative super-martingale
and therefore for any couple of stopping times $\tau_1<\tau_2$ we have $Z(\tau_1)\ge Z(\tau_2)$.
This implies that if $\tau=\inf\{t>0 : |Z(t)|=0\}$ we have that $Z(t)=Z(\tau)$, $\tilde{\P}$-a.s. for $t>\tau$.
On the other hand, by \eqref{e4.22} and \eqref{e4.16} it follows that
\[
\E Z(t)+\alpha\int_0^t e^{-\tilde C\,s}\tilde{\P}(\tau>s)\,ds\le|g(x)|+
\tilde C\int_0^t\E Z(s)\,ds,\qquad \forall t\ge 0
\]
and therefore
\[
\tilde{\P}(\tau>t)\le \frac{\tilde C}{\alpha}\,(1-e^{-\tilde C\,t})^{-1}\,|g(x)|,\qquad \forall t>0,
\]
which is just \eqref{e4.16a}. This shows that $\tilde X(t)$ reaches the manifold $\Sigma$ in stopping time $\tau$ and remains there for $t>\tau$ with a probability $\tilde{\P}$ greater or equal $\frac{\tilde C}{\alpha}\,(1-e^{-\tilde C\,t})^{-1}\,|g(x)|$. The proof is complete.
\end{proof}
\renewcommand{\proofname}{Proof}
\begin{Remark}\em
If conditions \eqref{e4.14}, \eqref{e4.15} are satisfied with $\alpha=0$ in Theorem \ref{t4.2} then
only part (i) follows.
\end{Remark}
Theorem \ref{t4.2} can be used to design feedback controllers for stochastic differential systems
with a sliding mode dynamics on a given surface $\Sigma=\{x: g(x)=0\}$. Such an example is presented below.
\begin{Example}\em
Consider the controlled stochastic second order system
\begin{equation}
\ddot X +a_1\,\dot X=\sigma_0(X,\dot X)\dot\beta+u\qquad  \mbox{in }(0,\infty).
\end{equation}
We assume that $\sigma_0\in \mbox{Lip}(\R^2)$. 

Our aim is to find a feedback controller $u=-f_0(X,\dot X)$ such that the corresponding
closed loop system
\begin{equation}\label{e4.25}
\begin{array}{l }
\ddot X +a_1\,\dot X+f_0(X,\dot X)=\sigma_0(X,\dot X)\dot\beta \\
X(0)=x_0,\quad \dot X(0)=x_1
\end{array}
\end{equation}
has the sliding mode equation
\begin{equation}\label{e4.26}
a_2 X+\dot X=0,
\end{equation}
for some $a_2\in\R$. Here $\beta$ is a Brownian motion in a probability space $(\Omega,\mathcal F,\P)$
and $\dot\beta$ is the associated white noise.

We choose
\begin{equation}\label{e4.27}
f_0(r_1,r_2)=\alpha\,\mbox{sgn}(a_2\,r_1+r_2),\qquad \forall(r_1,r_2)\in\R^2
\end{equation}
where $\alpha>0$  and rewrite equation \eqref{e4.25} as
\begin{equation}\label{e4.28}
\begin{array}{ l}
dX_1-X_2\,dt=0 \\
dX_2+a_1\,X_2\,dt+\alpha\,\mbox{sgn}(a_2\,X_1+X_2)\,dt=\sigma_0(X_1,X_2)\,d\beta
\end{array}
\end{equation}
for $t\ge 0$, where as us usually $\mbox{sgn}(u)=\frac{u}{|u|}$ for $u\not=0$.\\
Equation \eqref{e4.28} is a ``jump'' system of the form \eqref{e1.3} where
\[
f(r_1,r_2)=\left(\begin{array}{c}
-r_2\\
a_1\,r_2+\alpha\,\mbox{sgn}(a_2\,r_1+r_2)
\end{array}\right),\qquad \forall (r_1,r_2)\in\R^2,
\]
\[
\sigma(r_1,r_2)=\left(\begin{array}{c}
0\\
\sigma_0(r_1,r_2)
\end{array}\right),\qquad \forall (r_1,r_2)\in\R^2,
\]
and so $f$ is of the form \eqref{e2.11n} where
\[
f_1(r)=\left(\begin{array}{c}
-r_2\\
a_1\,r_2+\alpha
\end{array}\right),\quad f_2(r)=\left(\begin{array}{c}
-r_2\\
a_1\,r_2-\alpha
\end{array}\right),\qquad r=(r_1,r_2)\in \R^2
\]
\[
g(r)=a_2\,r_1+r_2,\qquad r=(r_1,r_2).
\]
It is easily seen that conditions \eqref{e4.12}--\eqref{e4.16} hold and so Theorem \ref{t4.2} is applicable
to the present case. We get
\begin{Corollary}
The stochastic closed loop system \eqref{e4.28}, equivalently \eqref{e4.25}, \eqref{e4.27},
has 
the ``sliding mode" \eqref{e4.26}. More precisely, for every
$(x_0,x_1)\in \R^2$ there is a martingale solution $(X_1(t),X_2(t))$ which reaches the surface
$\Sigma=\{(x_1,x_2): a_2\,x_1+x_2=0\}$ in time $t$ with a probability $\ge 1-(\alpha\,t)^{-1}|a_2\,x_0+x_1|$,
and remains $\tilde{\P}$-a.s. on this surface after that time.
\end{Corollary}
This describes a typical ``sliding-mode" behaviour for solutions $X$ to \eqref{e4.25}, namely
\[
a_1\,X(t)+\dot X(t)=0
\]
on $(t_0,\infty)\times\Omega_0$ where $\tilde{P}(\Omega_0)\ge 1-(\alpha\,t_0)^{-1}|a_2\,x_0+x_1|$.
(We refer to \cite{H}, \cite{SX}, \cite{SK},  for references and other significant results on ``sliding-mode" behaviour of
stochastic differential systems).
\end{Example}

\section{Existence of a weak solution to heat equation \eqref{e1.4}}
The following hypotheses will be assumed throughout in the sequel.
\begin{description}
  \item[i)] $f\in L^\infty_{\text{\tiny loc}}(\R)$ and $|f(r)|\le a_1 |r|+b_1$, $\forall r\in\R$
  \item[ii)] $W$ is the cylindrical Wiener process \eqref{e1.2} where $\{e_j\}_{j=1}^\infty$
  is an orthonormal basis in $L^2(\Omega)$ given by $-\Delta e_j=\lambda_je_j$ in $\O$;
  $e_j=0$ on $\bO$ and
  \begin{equation}\label{e2.1}
  \sum_{j=1}^\infty \mu_j^2\,\lambda_j^2>\infty
  \end{equation}
  \item[iii)] $b\in C^2(\R)\cap $Lip$(\R)$.
\end{description}

\begin{Definition}\label{d3.1}
Let $x\in L^2(\O)$. We call weak (martingale) solution to \eqref{e1.1} a tuple
$(\Omega,\mathcal F,(\mathcal F_t)_{t\ge 0},\P,W,X)$, where $(\Omega,\mathcal F,(\mathcal F_t)_{t\ge 0},\P)$
is a filtered probability space where there are defined a $(\mathcal F_t)_{t\ge 0}$-Wiener process $W$
and a continuous $(\mathcal F_t)_{t\ge 0}$-adapted $L^2(\O)$-valued process $X=(X(t))_{t\ge0}$
such that, $\P$-a.s.,
\begin{equation}\label{e2.2}
X(t)=e^{-tA}x+\int_0^t e^{-(t-s)A}\eta(s)\,ds+\int_0^t e^{-(t-s)A}\,b(X(s))\,dW(s),
\end{equation}
where $\eta\in L^\infty(\ter)$ is a $(\mathcal F_t)_{t\ge 0}$-adapted process such that
\begin{equation}\label{e2.3}
\eta\in F(X), \mbox{a.e in } \ter.
\end{equation}
\end{Definition}
Here $A=-\Delta$ with $D(A)=H_0^1(\O)\cap H^2(\O)$ and $e^{-A\,t}$ is the $C_0$-semigroup on
$L^2(\O)$ generated by $-A$.

We note that the linear operator $b(X)$ arising in \eqref{e2.2} is defined by
\[
b(X) h=\sum_{j=1}^\infty \mu_j\,b(X) \<h,e_j>_2\,e_j,\qquad \forall h\in L^2(\O).
\]

\paragraph{The construction of a weak (martingale) solution.}
We consider the approximating equation
\begin{equation}\label{e2.4}
\begin{array}{l}
dX_\eps-\Delta X_\eps\,dt+f_\eps(X_\eps)\,dt=b(X_\eps)\,dW,\quad \text{in }(0,T)\times\O\\
X_\eps=0,\hspace{6.05cm}\text{on }(0,T)\times\bO\\
X_\eps(0,\xi)=x(\xi),\hspace{4.65cm}\xi\in \O
\end{array}
\end{equation}
where $\eps>0$ and, as in the finite dimensional case (see \eqref{e4.7}),
\begin{equation}\label{e2.5}
f_\eps(r)=\frac{1}{\eps}\int_{-\infty}^{\infty}f(s)\,\rho(\tfrac{r-s}\eps)\,ds=
\int_{-\infty}^{\infty}f(r-\eps\,\theta)\,\rho(\theta)\,d\theta,\quad \forall r\in \R.
\end{equation}
Here $\rho\in C^\infty_{\text{\tiny0}}(\R)$ is such that
\begin{equation}\label{e2.6}
\rho(\theta)\ge0\quad\rho(\theta)=\rho(-\theta),\quad \rho(\theta)=0\mbox{ for }|\theta|\ge 1,
\quad \int_{-\infty}^\infty\rho(\theta)\,d\theta=1.
\end{equation}
Clearly by (i) we have
\begin{equation}\label{e2.7}
f_\eps\in C^1(\R),\qquad |f_\eps(r)|\le a_1\,|r|+b_1+a_1\,\eps,\quad \forall r\in\R, \eps>0.
\end{equation}
By standard existence theory for infinite dimensional stochastic equations with Lipschitz nonlinearity
it follows that \eqref{e2.4} has a unique strong solution
\begin{equation}
X_\eps\in L^2(\Omega;C([0,T];L^2(\O)))\cap L^2(\Omega;L^2(0,T;H_0^1(\O))),
\end{equation}
see \cite[pag.45]{DZ}. By It\^o's formula we get $\P$-a.s.
\[
\begin{split}
\tfrac12\,|X_\eps(t)|_2^2+\int_0^t\|X_\eps(s)\|_1^2\,ds+
\int_0^t \langle f_\eps(X_\eps(s)),X_\eps(s)\rangle_2\,ds\hspace{3cm}\\
=\tfrac12\,|x|_2^2+\tfrac12\int_0^t\sum_{j=1}^\infty\mu_j^2| b(X_\eps(s))\,e_j|_2^2\,ds
+\int_0^t\langle b(X_\eps(s))dW(s),X_\eps(s)\rangle_2\,ds,\\
\forall t\in [0,T],
\end{split}
\]
and so by the Burkholder-Davis-Gundy formula we obtain by some calculation
involving (i)--(iii)
\begin{equation}\label{e2.9}
\E\sup_{t\in[0,T]}|X_\eps(t)|_2^2+\E\int_0^t\|X_\eps(s)\|_1^2\,ds\le C\,
(|x|_2^2+1),\quad \forall\eps>0,
\end{equation}
where $C$ is independent of $\eps$. By \eqref{e2.7} we also have
\[
\E\sup_{t\in[0,T]}|f_\eps(X_\eps(t))|_2^2\le C\,(|x|_2^2+1).
\]
Then on a subsequence, again denoted in the same way, we have for $\eps\to0$
\begin{equation}\label{e2.10}
\begin{split}
X_\eps\to X\quad \mbox{weak-star in }L^\infty(0,T;L^2(\Omega;L^2(\O))\\
\mbox{weakly in }L^2(0,T;L^2(\Omega;H_0^1(\O))
\end{split}
\end{equation}
\begin{equation}
f_\eps\to \eta\quad \mbox{weak-star in } L^\infty((0,T);L^2(\Omega;L^2(\O)))
\end{equation}
\begin{equation}\label{e2.12}
X_\eps(t)\to X(t)\quad \mbox{weakly in } L^2(\Omega;L^2((0,T)\times\O))
\end{equation}
\[
b(X_\eps)\to b^\ast\quad \mbox{weakly in }L^2(\Omega;L^2((0,T)\times\O))
\]and
\begin{equation}
\begin{array}{l}
dX-\Delta X\,dt+\eta\,dt=b^\ast\,dW\quad \mbox{in } (0,T)\times\O,
\\
X(0)=x\hspace{4.4cm}\mbox{in }\O,
\\
X=0\hspace{5cm}\mbox{on } (0,T)\times\bO,
\end{array}
\end{equation}
that is
\[
X(t)-\int_0^t\Delta X(s)\,ds+\int_0^t\eta(s)\,ds=
x+\int_0^t b^\ast(s)\,dW(s),\quad \forall t\in [0,T],\;\P\mbox{-a.s.}
\]
Assume now that $x\in H_0^1(\O)$. Then by an application of  It\^o's formula in \eqref{e2.4}
to the function $x\to \tfrac12\,\|x||_1^2$ we get for some $C_1$, $C_2\ge 0$,
\[
\begin{split}
\tfrac12\,\|X_\eps(t)\|_1^2+\int_0^t |\Delta X_\eps(s)|_2^2\,ds\le \tfrac12\,\|x\|_1^2+
C_1\int_0^t(|\Delta X_\eps(s)|_2^2+|X_\eps(s)|_2^2)\,ds\\
+\tfrac12\int_0^t \sum_{j=1}^\infty\mu_j^2\,\|\sigma(X_\eps)\,e_j\|_1^2\,ds+
\int_0^t \langle \Delta b(X_\eps(s)),dW(s)\rangle\ds+C_2
\end{split}
\]
and in virtue of (ii), (iii)  this yields via Burkholder-Davis-Gundy formula
\begin{equation}\label{e2.14}
\E\sup_{t\in[0,T]}\|X_\eps(t)\|_1^2+\E\int_0^t |\Delta X_\eps(s)|_2^2\,ds\le
C(\|x\|_1^2+1),\quad \forall\eps>0.
\end{equation}
(Everywhere in the sequel we shall denote by $C$ several constants
independent of $\eps$.)\\
Then, in this case besides \eqref{e2.10}, \eqref{e2.12}, we also have
\begin{equation}
X\in C([0,T];L^2(\O))\cap L^2(\Omega;L^\infty(0,T;H_0^1(\O)))\cap L^2(\Omega;L^2(0,T;H^2(\O))).
\end{equation}
Since the weak convergences \eqref{e2.10}-\eqref{e2.12} are not sufficient
to conclude that \eqref{e2.3} holds, then proceeding as in the proof of Theorem
\ref{t2.1} we shall replace $\{X_\eps\}$ by a sequence
$\{\tilde X_\eps\}$ of processes defined in a probability space
$\{\tilde\Omega,\tilde{\mathcal F},\tilde\P,\tilde W\}$ such that
$\mathcal L(X_\eps)=\mathcal L(\tilde X_\eps)$ where $\mathcal L$ is the law of the process.

To this end, consider the sequence $\{\nu_\eps\}_{\eps\ge0}$ of probability measures,
$\nu_\eps=\mathcal L(X_\eps)$, that is $\nu_\eps(B)=\P(X_\eps\in B)$
for any Borelian set $B\subset C([0,T];L^2(\O))$. We have
\begin{Lemma}
Let $x\in H_0^1(\O)$. Then the sequence $\{\nu_\eps\}_{\eps>0}$ is tight in the space $C([0,T];L^2(\O))$.
\end{Lemma}
\begin{proof}
This means that for each $\delta>0$ there is a compact subset $B$ of $C([0,T];L^2(\O))$
such that $\nu_\eps(B^c)\le \delta$ for all $\eps>0$. We take for $r>0$, $\gamma>0$,
\[
\begin{split}
B=B_{r,\gamma}=\{y\in C([0,T];L^2(\O)): |y(t)|_2\le r, \forall t\in[0,T],\hspace{2cm}\\
\|y\|_{L^\infty(0,T;H_0^1(\O))}\le r,|y(t)-y(s)|_2\le \gamma \,|t-s|^{\tfrac12},\quad \forall t,s\in[0,T]\}
\end{split}
\]
Clearly, by Ascoli-Arzel\`a theorem, $B_{r,\gamma}$ is compact in $C([0,T];L^2(\O))$.
On the other hand, by \eqref{e2.4} we have via It\^o's formula applied to the process
$t\to|X_\eps(t)-X_\eps(s)|_2^2$
\[
\begin{split}
\tfrac12\,\E |X_\eps(t)-X_\eps(s)|_2^2+\E\int_s^t\langle \nabla X_\eps(\theta),\nabla (X_\eps(\theta)-X_\eps(s)) \rangle_2\,d\theta\\+
\E\int_s^t\langle f_\eps(X_\eps(\theta)),X_\eps(\theta)-X_\eps(s)\rangle_2\,d\theta\le C\,\E\int_s^t|X_\eps(\theta)|_2^2\,d\theta
\quad 0\le s\le t\le T.
\end{split}
\]
Taking into account estimates \eqref{e2.7}, \eqref{e2.9} we obtain via Gronwall's lemma that
\begin{equation}\label{e2.16}
\E|X_\eps(t)-X_\eps(s)|_2^2\le C\int_s^t(|X_\eps(\theta)|_2^2+|\nabla X_\eps(\theta)|_2^2)\,d\theta\le C\;|t-s|.
\end{equation}
By estimates \eqref{e2.9}, \eqref{e2.14}, \eqref{e2.16} and taking into account that
\[
\rho\;\P[|Y|\ge\rho] \le \E|Y|,\quad \forall \rho>0,
\]
we infer that there are $\gamma$, $r$ independent of $\eps$ such that
$\nu_\eps(B_{r,\gamma}^c)\le\delta$, as desired.
\end{proof}
Then by the Skorohod theorem (see, e.g., Theorem 2.4 in \cite{DZ}) there are a probability space
$(\tilde\Omega,\tilde{\mathcal F},\tilde\P)$ and the stochastic processes $\tilde X$, $\{\tilde X_\eps\}_{\eps>0}$
on $(\tilde\Omega,\tilde{\mathcal F},\tilde\P)$ such that the law $\mathcal L(\tilde{X_\eps})$ of $\tilde X_\eps$ coincides with $\mathcal L(X_\eps)$
and $\P$-a.s.
\begin{equation}\label{e2.17}
\tilde X_\eps\to \tilde X \quad \mbox{in } C([0,T];L^2(\O))
\end{equation}
as $\eps\to 0$. We have also $\mathcal L(X)=\mathcal L(\tilde X)$.
Since
$\mathcal L(f_\eps(X_\eps))=\mathcal L(f_\eps(\tilde{X_\eps}))$,
$\mathcal L(\sigma(X_\eps))=\mathcal L(\sigma(\tilde{X_\eps}))$
by \eqref{e2.17} and \eqref{e2.5} we see that
\begin{equation}
\begin{array}{l}
f_\eps(\tilde X_\eps)\to \tilde \eta,\\
b(X_\eps)\to b(\tilde X),
\end{array}
\qquad \mbox{a.e. in } (0,T)\times\O\times\tilde\Omega,
\end{equation}
where $\mathcal L(\tilde\eta)=\mathcal L(\eta)$ and
\begin{equation}\label{e2.19}
\tilde \eta\in F(\tilde X),\qquad \mbox{a.e. in } (0,T)\times\O\times\tilde\Omega.
\end{equation}
The latter follows as in the proof of Theorem \ref{t2.1} taking
into account that in this case $F$ is given by \eqref{e1.5}, but we omit the details.\\
We set 
\begin{equation}\label{eq:3.20}
M_\eps(t)=X_\eps(t)-x-\int_0^t\Delta X_\eps(s)\,ds+\int_0^t f_\eps(X_\eps(s))\,ds,\quad t\in [0,T]
\end{equation}
and
\begin{equation}
\tilde M_\eps(t)=\tilde X_\eps(t)-x-\int_0^t\Delta \tilde X_\eps(s)\,ds+\int_0^t f_\eps(\tilde X_\eps(s))\,ds,\quad t\in [0,T].
\end{equation}
It turns out that $\tilde M_\eps$ is a square integrable martingale on $(\tilde\Omega,\tilde{\mathcal F},\tilde\P)$
with respect to the filtration $\mathcal F_t=\sigma\{ \tilde X_s; s\le t\}$
because since $\mathcal L(\tilde M_\eps)=\mathcal L(M_\eps)$ and $M_\eps$
is a square integrable martingale on $(\Omega,\mathcal F,\P)$ we have
\begin{equation}
\E\big[ (\tilde X_\eps(t)-\tilde X_\eps(s)-\int_s^t\Delta \tilde X_\eps(\theta)\,d\theta+
\int_s^t f_\eps(\tilde X_\eps(\theta))\,d\theta)\chi(\tilde X_\eps)\big]=0
\end{equation}
for any bounded continuous function $\chi$ and all $0\le s\le t\le T$.\\
Passing to the limit in \eqref{eq:3.20} and taking into account \eqref{e2.17}--\eqref{e2.19} one
obtains that the process
\[
\tilde M(t)=\tilde X(t)-x-\int_0^t \Delta \tilde X(s)\,ds+\int_0^t\tilde \eta(s)\,ds,\quad t\ge0
\]
is an $L^2(\O)$-valued martingale with respect to filtration
$\tilde{\mathcal F}_t=\sigma\{\tilde X(s), s\le t\}$, $t\in [0,T]$,
with finite quadratic variation, see \cite[pag.234]{DZ}. Then by the representation theorem 8.2 in \cite{DZ}
there is a larger probability space $(\tilde{\tilde \Omega},\tilde{\tilde{\mathcal F}},\tilde{\tilde \P})$,
a filtration $\{\tilde{\tilde{\mathcal F}}_t\}_{t\ge 0}$ and an $L^2(\O)$-cylindrical Wiener process
$\tilde{\tilde W}(t)$ on it such that, $\tilde{\tilde\P}$-a.s.,
\[
\tilde M(t)=\int_0^t b(\tilde X(s))\,d\tilde{\tilde W}(t),\qquad t\in[0,T].
\]
This means that the system $(\tilde{\tilde \Omega},\tilde{\tilde{\mathcal F}},
\{\tilde{\tilde{\mathcal F}}_t\}_{t\ge 0},\tilde{\tilde \P},\tilde{\tilde W}(t),\tilde X(t))$
is a martingale solution to \eqref{e1.1}. We have proved therefore
 \begin{theorem}\label{t2.2}
Under Hypotheses {\em (i), (ii)}, for each $x\in H_0^1(\O)$, there is at least one martingale solution
$(\tilde{\tilde\Omega},\tilde{\tilde{\mathcal F}},\{\tilde{\tilde{\mathcal F}}_t\}_{t\ge0},\tilde{\tilde\P},\tilde{X})$
to equation \eqref{e1.1} and $\tilde X$ is given by \eqref{e2.17}. Moreover, we have
\begin{equation}\label{e2.23}
\tilde X\in L^2(\tilde{\tilde\Omega};L^\infty(0,T;H_0^1(\O)))\cap L^2(\tilde{\tilde\Omega};L^2(0,T;H^2(\O))).
\end{equation}
\end{theorem}
\noindent We note that \eqref{e2.23} follow by \eqref{e2.14} and \eqref{e2.17}.
\begin{Remark}\em
Under  additional assumptions on $b$ (for instance if it is independent of $X$)
it turns out that the martingale solution $\tilde X$ is the unique strong solution, see \cite{DFPR}.
(See also Remark 2.1)
\end{Remark}

\section{Sliding mode control of the stochastic heat equation}
For parabolic stochastic equations of the form \eqref{e1.1}
a ``sliding'' mode dynamic arises for discontinuous (``jump'') functions $f\colon\R\to\R$ of the form
\eqref{e2.11n}, that is
\begin{equation}\label{e3.1}
f(r)=\begin{cases}f_1(r)& \text{for } g(r)>0\\f_2(r))& \text{for } g(r)<0\end{cases},
\qquad r\in\R
\end{equation}
where $g$, $f_1$, $f_2$ are given continuous functions.

As in the previous finite dimensional case, the objective of the ``sliding-mode" control is to design for the linear time invariant system
\begin{equation}\label{e44.2}
\begin{array}{l}
dX-\Delta\,X\,dt=du\\
X=0\\X(0)=x
\end{array}\quad
\begin{array}{l}
\mbox{in } (0,T)\times\O\\
\mbox{on } (0,T)\times\bO\\
\mbox{in }\O
\end{array}
\end{equation}
a stochastic feedback controller of the form
\begin{equation}\label{e44.3}
du=-f(X)\,dt +b(X)\,dW
\end{equation}
such that the ``sliding" motion occurs on the manifold $\Sigma=\{X : g(X)=0\}$
which is also referred as ``sliding" or  ``switching" surface.
Roughly speaking, this means that any trajectory of the closed loop system \eqref{e44.2}-\eqref{e44.3},
which starts from initial state $x$, reaches the sliding surface $\Sigma$ at a certain time $t_0$ and
remains there for $t\ge t_0$. As a matter of fact, this last phase of the dynamics is called ``sliding mode".

Of course in virtue of Theorem \ref{t2.2} a weak solution $X$ to \eqref{e44.2} in the
sense of Definition \ref{d3.1} exists for the extended multivalued closed loop system
\begin{equation}\label{e44.4}
\begin{array}{l}
\begin{aligned}dX-\Delta\,X\,dt+f_1(X)\tilde H(g(X))\,dt\hspace{2cm}\\+f_2(X)\tilde H(-g(X))\,dt=\sigma(X)\,dW\end{aligned}\\
X=0\\X(0)=x
\end{array}\quad
\begin{array}{l}
\begin{aligned}\\\mbox{in } (0,T)\times\O\end{aligned}\\
\mbox{on } (0,T)\times\bO\\
\mbox{in }\O.
\end{array}
\end{equation}
Here $\tilde H$ is the multivalued Heaviside function \eqref{eq:2.22} on $\R$.
To begin with we shall prove first an invariance result for the manifold $\Sigma=\{X:g(X)=0\}$.
\begin{theorem}\label{t3.1}
Let $g\in C^2
(\R)$, $f_1$, $f_2$ be continuous functions which satisfy assumption {\em(i)}
and let $b$, satisfying {\em (iii)}, be such that
\begin{align}
\label{e44.5}&b^2(r)(g\,g''+(g')^2)(r)\le C\,g^2(r),&\forall r\in\R\hspace{.9cm}\\
&g(r)\,g''(r)+(g'(r))^2\ge0&\forall r\in\R\hspace{.9cm}\\
\label{e3.7}&f_1(r)\,g'(r)\ge0&\mbox{for }g(r)>0\\
\label{e3.8}&f_2(r)\,g'(r)\le0&\mbox{for }g(r)<0.
\end{align}
for some $C>0$. 
Then, for all $x\in H_0^1(\O)$ such that $g(x)=0$ on $\O$, there is a martingale solution $(\tilde{\tilde\Omega},\tilde{\tilde{\mathcal F}},\{\tilde{\tilde{\mathcal F}}_t\}_{t\ge0},\tilde{\tilde\P},\tilde{X})$ to system \eqref{e44.4} such that
\begin{equation}
g(\tilde X(t))=0,\qquad \forall t\in[0,T], \;\tilde{\tilde\P}\mbox{-a.s.}.
\end{equation}
\end{theorem}
\begin{proof} We start with the approximating equation \eqref{e2.4}.
We apply the It\^o formula to function $x\to g^2(x)$ and get
\[
\begin{split}
d\int_{\O} g^2(X_\eps(t,\xi))\,d\xi+\,2\int_{\O}(g\,g''+(g')^2)(X_\eps(t,\xi))|\nabla X_\eps(t,\xi)|^2\,d\xi\;dt\\
+\,2\int_{\O}f_\eps(X_\eps(t,\xi))g(X_\eps(t,\xi))g'(X_\eps(t,\xi))\,d\xi\;dt=\\
\sum_{j=1}^\infty \mu_j^2\int_{\O}|b(X_\eps(t,\xi))\,e_j|^2|(g\,g''+(g')^2)(X_\eps(t,\xi))|\,d\xi\;dt\\
+\sum_{j=1}^\infty \mu_j\int_{\O}b(X_\eps(t,\xi))g(X_\eps(t,\xi))g'(X_\eps(t,\xi))\,e_j\,d\xi\;d\beta_j(t)
\end{split}
\]
Taking into account \eqref{e2.5},\eqref{e2.7}, we obtain that
\[
\begin{split}
\int_\O f_\eps(&X_\eps)g(X_\eps)g'(X_\eps)\,d\xi=\int_{}^{}\rho(\theta)\\
&\Big(\int\limits_{[g(X_\eps-\eps\theta)>0]}f_1(X_\eps-\eps\theta)(g g')(X_\eps-\eps\theta)\,d\xi+\\
&\hspace{0cm}\int\limits_{[g(X_\eps-\eps\theta)<0]}f_2(X_\eps-\eps\theta)(g g')(X_\eps-\eps\theta)\,d\xi\Big)\,d\theta+
\zeta_\eps(t),\qquad \forall t\in [0,T],
\end{split}
\]
where
\[
\zeta_\eps(t)\le \tilde C\,\eps \int_\O (|X_\eps(t,\xi)|+1)\,d\xi
\]
with $\tilde C= C_1\,C$, $C_1>0$; thus, by  \eqref{e44.5}--\eqref{e3.7},  this yields
\[
\begin{split}
\E\int_{\O} g^2(X_\eps(t,\xi))\,d\xi\le C\,\E\int_0^t\int_{\O} g^2(X_\eps(t,\xi))\,d\xi\,ds\hspace{3cm}\\
+\;\delta(\eps)\,\E \int_0^t\int_{\O} (X_\eps(t,\xi)^2+1)\,d\xi\,ds,\quad \forall t\in [0,T].
\end{split}
\]
where $\lim_{\eps\to0}\delta(\eps)=0$ and the constant $C$ is independent of $\eps$.\\
This yields via Gronwall's lemma,
\begin{equation}\label{e3.10}
\E\int_\O g^2(X_\eps(t,\xi))\,d\xi\le \delta(\eps)\;\exp(C\,t),\qquad \forall t\in [0,T].
\end{equation}
If $\tilde X_\eps$ is defined as in the proof of Theorem \ref{e2.1}, that is
$\mathcal L(\tilde X_\eps)=\mathcal L(X_\eps)$ and \eqref{e2.17} holds, we get by \eqref{e3.10} that
\[
\tilde \E\int_\O g^2(\tilde X_\eps(t,\xi))\,d\xi\le \delta(\eps)\,\exp(C\,t),\qquad \forall t\in [0,T],\;\forall \eps>0
\]
where $\tilde\E$ is the expectation in probability space $(\tilde\Omega,\tilde{\mathcal F},\tilde\P)$.
Hence, letting $\eps$ tend to zero we get   $g^2(\tilde X)=0$, $dt\times d\xi\times\tilde{\tilde\P}$-a.e.
in $(0,T)\times\O\times \tilde{\tilde\Omega}$ as claimed.
\end{proof}
\begin{Remark}\em
In the particular case where the function
\[
F(r)\equiv f_1(r)H(g(r))+f_2(r)H(-g(r)),\quad r\in\R
\]
is a maximal monotone graph in $\R\times\R$ with $R(F)=D(F)=\R$, equation \eqref{e44.4} has
a unique strong solution $X$. This happens for instance if $f_i$, $i=1,2$, are
monotonically nondecreasing continuous functions such that $f_1\ge f_2$ on $\R$
and $g(x)=x$ (see \cite{B1}). Then the corresponding system \eqref{e44.4}
\begin{equation}\label{e4.11}
\begin{array}{l}
\begin{aligned}&dX-\Delta X\,dt+(f_1(X)H(X)+f_2(X)H(-X))\,dt=\\
&\hspace{6.2cm}b(X)\,dW, \quad \mbox{in } (0,T)\times\O\\
&X=0,\hspace{6.28cm}\qquad\mbox{on }(0,T)\times\bO,\end{aligned}
\end{array}\end{equation}
has the invariant manifold $X=0$.
\end{Remark}
Under stronger assumptions on $g$ and $b$ it turns out that the closed loop
system \eqref{e44.2}--\eqref{e44.3} (equivalently \eqref{e4.11}) has
a ``sliding'' mode dynamics with the switching manifold
$\Sigma=\{X:g(X)=0\}$.
Namely, we assume that\\
\begin{Hypothesis}\label{(k)}
$f_i$, $i=1,2$  satisfy assumption (i) and $g\in C^2_b(\R)$, $b\in C^2(\R)\cap Lip(\R)$
are such that
\begin{align}
\label{e44.12}&|g''(r)|\;|b(r)|^2\le C^\ast\,|g(r)|,\qquad \forall r\in \R,\\
\label{e44.13}
&f_1(r)\,g'(r)\ge \alpha\quad \mbox{if } g(r)>0; f_2(r)\,g'(r)\le -\alpha\quad \mbox{if } g(r)<0\\
\label{e44.14}&g', g'' \in L^\infty(\R),\quad g''\,\mbox{sgn}(g)\ge0\quad \mbox{on }\R,
\end{align}
where $\alpha>0$.
\end{Hypothesis}

We note that by Theorem \ref{t2.2}, equation \eqref{e4.11} has a martingale solution $\tilde X$ given by \eqref{e2.17}.
\begin{theorem}\label{t5.1}
Under Hypothesis \ref{(k)}, for each $x\in H_0^1(\O)$ there is a martingale solution
$(\tilde{\tilde \Omega},\tilde{\tilde {\mathcal F}},\tilde{\tilde \P},\tilde{\tilde W},\tilde X)$ to \eqref{e44.2}, \eqref{e44.3}
such that for $\tau=\inf\{t : |g(\tilde X(t))|=0\}$ we have
\begin{equation}\label{e44.15}
\tilde{\tilde\P}(\tau>t)\le \frac{\tilde C}{\alpha}(1-e^{-\tilde C\,t})^{-1}\,|g(x)|^2.
\end{equation}
where $\tilde C=C_1C^\ast$. Moreover, if $g(x)=0$ a.e. in $\O$, then $g(\tilde X(t))=0$ for all $t\ge0$.
\end{theorem}
\begin{proof} The proof is very similar to that of Theorem \ref{t4.2}, so it will be sketched only. If
$X_\eps$ is the solution to equation \eqref{e2.4} and $\tilde X_\eps$ such that $\mathcal L(X_\eps)=\mathcal L(\tilde X_\eps )$,
we get via It\^o's formula applied to function
$x\to \varphi_\lambda(g(x))$ and after letting $\lambda\to0$
\[
\begin{split}
d\,|g(\tilde X_\eps(t))|_2+\alpha \mathds 1_{[|g(\tilde X_\eps(t))|>0]}dt\le
\langle b(\tilde X_\eps(t))dW(t),g'(\tilde X_\eps(t))\,\mbox{sgn}(g(\tilde X_\eps(t))\rangle_2\\
+\tfrac12\sum_{k=1}^\infty \mu_k^2 \int_O |b(X_\eps)\,e_k|^2|g''(X_\eps)|\,d\xi
\end{split}
\]
where $\varphi_\lambda$ is the function introduced in the proof of Theorem \ref{t4.2} (see \eqref{e4.17}).
We have used here condition \eqref{e44.14} which in virtue of \eqref{e4.17} yields
\begin{equation}\label{e44.16}
\begin{split}
\lim_{\lambda\to0}
\int_{\O}\Delta\tilde X_\eps(t,\xi))\cdot\varphi_\lambda'(g(\tilde X_\eps(t,\xi)))g'(\tilde X_\eps(t,\xi))\,d\xi\\
=-\lim_{\lambda\to0}\int_{\O}\nabla\tilde X_\eps(t,\xi))\cdot\nabla
\big(\varphi_\lambda'(g(\tilde X_\eps(t,\xi)))&g'(\tilde X_\eps(t,\xi))\big)\,d\xi\le0,
\end{split}
\end{equation}
and also assumption \eqref{e44.13} which, as we have seen in the proof of Theorem \ref{t4.2},
implies that
\[
f_\eps(\tilde X_\eps)\,g'(\tilde X_\eps)\,\text{sgn}\,g(\tilde X_\eps)\ge \alpha-\delta(\eps)(1+|\tilde X_\eps|).
\]
Now using \eqref{e44.12} and  letting $\eps\to0$ we get
\[
\begin{split}
|g(\tilde X(t))|_2+\alpha \int_s^t\mathds 1_{[|g(\tilde X(\theta))|>0]}d\theta\le|g(\tilde X(s))|_2+
\tilde C\int_s^t |g(\tilde X(\theta))|_2\,d\theta+\\\int_s^t\langle b(\tilde X(\theta))dW(\theta),g'(\tilde X(\theta))\,\mbox{sgn}(g(\tilde X(\theta))\rangle_2
\qquad \mbox{for }0\le s\le t
\end{split}\]
for some constant $\tilde C>0$.\\
This yields (see \eqref{e4.22})
\[
\begin{split}
e^{-\tilde C\,t}|g(\tilde X(t))|_2+\alpha\int_s^t e^{-\tilde C\,\theta}\mathds1_{[|g(\tilde X(\theta))|>0]}\,d\theta\le
e^{-\tilde C\,s}|g(\tilde X(s))|_2+\hspace{1cm}\\
\int_s^t e^{-\tilde C\,\theta}\mathds1_{[|g(\tilde X(\theta))|>0]}
\,\langle g'(\tilde X(\theta))\mbox{sgn}(g(\tilde X(\theta))),\sigma(\tilde X(\theta))\,dW(\theta)\rangle_2,\\
\quad 0\le s\le t<\infty.
\end{split}
\]
Here $\tilde Z=|g(\tilde X(t))|_2\,e^{-\tilde C\,t}$ is a nonnegative supermartingale and so $\tilde Z(t)=\tilde Z(\tau)$ $\tilde{\tilde\P}$-a.s. for
$t>\tau=\inf\{t>0:|\tilde Z(t)|=0\}$.
Taking expectation, we get
\[
\E\tilde Z(t)+\alpha\int_0^t e^{-\tilde C\,s}\tilde{\tilde\P}[\tau>s]\,ds\le |g(x)|_2+\tilde C\,\int_0^t\E\tilde Z(s)\,ds
\]
which implies the desired estimate \eqref{e44.15}.
\end{proof}

By \eqref{e44.14} we see that $g$ is convex on $[g>0]$, concave on $[g<0]$ and
that $[r : g(r)=0]=[\alpha_1,\alpha_2]$ is a closed interval.\\
Hence the switching manifold $\Sigma=\{X:g(X)=0\}$ is of the form
\[
\Sigma=\{X\in L^2(\O):\alpha_1\le X\le \alpha_2\}
\]
and so under the assumptions \eqref{e44.12}--\eqref{e44.14} the closed loop
system \eqref{e44.2}--\eqref{e44.3} has for each $x\in H^1_0(\O)$
a martingale solution which 
reaches the set $\Sigma$
(that is the interval $[\alpha_1,\alpha_2]$) by time $t$ with probability
estimated by \eqref{e44.15} 
and remains in this interval after that time.
\section{``Sliding'' mode control of  a stochastic pa\-rabolic systems}
Consider here the parabolic system
\begin{equation}\label{e5.1n}
\begin{array}{ l}
dX-\Delta X\,dt+f_1(X,Y)\,dt=b_1(X,Y)\,dW_1,\quad\!\text{in } (0,T)\times\O\\
dY-\Delta Y\,dt+f_2(X,Y)\,dt=b_2(X,Y)\,dW_2,\quad\text{in } (0,T)\times\O\\
X(0)=x(\xi),\quad Y(0)=y(\xi),\qquad\qquad\qquad\quad\ \xi\in\O\\
X=Y=0,\qquad\qquad\qquad\qquad\qquad\qquad\qquad\ \ \text{on } (0,T)\times\bO
\end{array}
\end{equation}
where $f_i\in C(\R^2)$ satisfy assumption (i), $b_i\in C^2(\R)\cap \text{Lip}(\R)$, $i=1,2$
and $W_1$, $W_2$ are Wiener processes of the form \eqref{e1.2} in the space $H=L^2(\O)\times L^2(\O)$.
Let $g\colon \R^2\to \R$, $g\in C^2(\R^2)$ be given.

Arguing as in the proof of Theorem \ref{t2.2} it follows that for each
$(x,y)\in H^1_0(\O)\times H^1_0(\O)$, system \eqref{e5.1n} has a martingale solution $(\tilde X, \tilde Y)$
obtained as limit of solutions $({\tilde X}_\eps,{\tilde Y}_\eps)$ to corresponding
approximating system
\begin{equation}\label{e5.1b}
\begin{array}{ l}
d \Big(\!\!\begin{array}{c}X \\Y\end{array}\!\!\Big)+
\Big(\!\!\begin{array}{cc}-\Delta&0\\0&-\Delta\end{array}\!\!\Big)\Big(\!\!\begin{array}{c}X \\Y\end{array}\!\!\Big)\,dt
+F_\eps\Big(\!\!\begin{array}{c}X \\Y\end{array}\!\!\Big)\;dt =
\Big(\!\!\begin{array}{c}b_1(X,Y) \,dW_1\\b_2(X,Y)\,dW_2\end{array}\!\!\Big)  \\
\Big(\!\!\begin{array}{c}X \\Y\end{array}\!\!\Big)(0)= \Big(\!\!\begin{array}{c}x \\y\end{array}\!\!\Big) 
\end{array}
\end{equation}
where $F\colon\R^2\to\R^2$ is given by
\[
F=\begin{cases}
f_1  & \text{in  }[g>0], \\
f_2   & \text{in }[g<0].
\end{cases}
\]
and
\[
F_\eps(r_1,r_2)=\int_{\R^2}\rho(r-\eps\theta)\,F(\theta)\,d\theta,\quad (r_1,r_2)\in\R^2,\;\;\eps>0.
\]
Assume further that
\begin{equation}\label{e5.2}
|D^2g(r)|\big(|b_1(r)|^2+|b_2(r)|^2\big)\le C\,|g(r)|,\quad\forall r\in\R^2
\end{equation}
\begin{equation}\label{ e5.3}
f_1(r)\,g_{r_1}(r)\ge\alpha\quad\text{in }[r:g(r)>0]
\end{equation}
\begin{equation}\label{ e5.4}
f_2(r)\,g_{r_2}(r)\le -\alpha\quad\text{in }[r:g(r)<0]
\end{equation}
where $\alpha>0$ and $(g_{r_1},g_{r_2})=\nabla g$,
\begin{equation}\label{e5.5}
g_{r_1r_1}\;\text{sgn}\,g\ge 0,\quad (g_{r_1r_2}^2-g_{r_1r_1}g_{r_2r_2})\;\text{sgn}\,g\le 0.
\end{equation}
We have
\begin{theorem}
Under assumptions \eqref{e5.2}--\eqref{e5.5} for each $(x,y)\in H_0^1(\O)\times H_0^1(\O)$ there is
a martingale solution $\{\tilde{\tilde\Omega},\tilde{\tilde{\mathcal F}},\tilde{\tilde\P},\tilde{\tilde W},(\tilde X,\tilde Y)\}$ to \eqref{e5.1n} such that if $\tau$ is the stopping time
 $\tau=\inf\{t:g(\tilde X(t),\tilde Y(t))=0\}$ then
\begin{equation}
\tilde{\tilde\P}[\tau>t]\le \frac{C}{\alpha}(1-e^{-C\,t})^{-1}|g(x,y)|_{(L^2(\O))^2}
\end{equation}
for some constant $C>0$.
\end{theorem}
\noindent The proof is exactly the same as that of Theorem \ref{t4.2} 
where the approximating equation \eqref{e2.4} is replaced by \eqref{e5.1b}. We note that in this case
the corresponding inequality \eqref{e44.16} is a consequence of hypothesis \eqref{e5.5}. The details are omitted.

A particular example is
\[
g(r_1,r_2)=\alpha_1\,r_1+\alpha_2\,r_2,\qquad \forall r_1,r_2\in \R
\]
which, for $f_1$ and $f_2$ satisfying condition
\[
\begin{array}{ l}
\alpha_1\,f_1(r)\ge \alpha\qquad\text{in }\{\alpha_1\,r_1+\alpha_2\,r_2>0\}\\
\alpha_2\,f_2(r)   \le -\alpha\quad\;\text{in }\{\alpha_1\,r_1+\alpha_2\,r_2<0\}
\end{array}
\]
where $\alpha>0$, imply that system  \eqref{e5.1n} has a martingale solution
$(\tilde X, \tilde Y)$ that reaches the linear manifold
\[
 \Sigma=\{\alpha_1\,\tilde X+\alpha_2\,\tilde Y=0\}
\]
in a time $t$ with probability $\tilde{\tilde\P}$ $\ge 1-C\,t^{-1}|\alpha_1\,x+\alpha_2\,y|_{(L^2(\O))^2}$
and remains on this manifold afterwards.

\end{document}